\documentclass[a4paper,12pt]{article}
 \usepackage[top=2.5cm,bottom=2.5cm,left=2.5cm,right=2.5cm]{geometry}
 \usepackage{cite, amsmath, amssymb, amsthm, amsfonts, color, verbatim, emptypage}
 \usepackage{enumerate}
 \usepackage[margin=1cm,%
 font=small,%
 format=hang,%
 labelsep=period,%
 labelfont=bf]{caption}

 \usepackage[para]{footmisc}

 \usepackage{pgf,tikz,pgfplots}
 \usepackage{mathrsfs}
 \usetikzlibrary{arrows}
 \usetikzlibrary{decorations.pathreplacing}
 \tikzstyle{vertex}=[circle,fill=black,inner sep=0pt]
\tikzstyle{edge} = [draw,thick]
\tikzstyle{weight} = [font=\small,fill=white, inner sep=0pt, opacity=0.0, text opacity=1]

 \date{ }
  
 \normalsize

 \newtheorem{theorem}{Theorem}
 
 \newtheorem{prop}[theorem]{Proposition}

\newtheoremstyle{mytheoremstyle1} 
        {\topsep}                    
        {\topsep}                    
        {}                   
        {}                           
        {\fontfamily{thm}\selectfont\scshape\color{black}}                   
        {.}                          
        {.5em}                       
        {}  
\theoremstyle{mytheoremstyle1}
 
 \newtheorem{case}{Case}

 \tikzstyle{every node}=[circle, draw, fill=black,inner sep=0pt, minimum width=4pt]
 
 \numberwithin{subcase}{case}
 \theoremstyle{definition}
 \newtheorem{remark}{Remark}

 \usepackage{authblk}

\title{\bf \vskip 3cm Tight Bounds on the Chromatic Edge Stability Index of Graphs}
\author[1]{Saieed Akbari\footnote{s\_akbari@sharif.edu}}
\author[2]{John Haslegrave\footnote{j.haslegrave@lancaster.ac.uk}}
\author[1]{Mehrbod Javadi\footnote{mehrbod.javadi@sharif.edu}}
\author[1]{Nasim Nahvi\footnote{nasim.nahvi1998@gmail.com}}
\author[1]{Helia Niaparast\footnote{h.niaparast@sharif.edu}}
\affil[1]{\small Department of Mathematical Sciences, Sharif University of Technology, Tehran, Iran}
\affil[2]{\small Department of Mathematics and Statistics, Lancaster University, UK}
\date{}							

\begin{document}
\maketitle
\begin{sloppypar}

\begin{abstract}

The chromatic edge stability index $\mathrm{es}_{\chi'}(G)$ of a graph $G$ is the minimum number of edges whose removal results in a graph with smaller chromatic index. We give best-possible upper bounds on $\mathrm{es}_{\chi'}(G)$ in terms of the number of vertices of degree $\Delta(G)$ (if $G$ is Class 2), and the numbers of vertices of degree $\Delta(G)$ and ${\Delta(G)-1}$ (if $G$ is Class 1). If $G$ is bipartite we give an exact expression for $\mathrm{es}_{\chi'}(G)$ involving the maximum size of a matching in the subgraph induced by the vertices of degree $\Delta(G)$. Finally, we consider whether a minimum mitigating set, that is a set of size $\mathrm{es}_{\chi'}(G)$ whose removal reduces the chromatic index, has the property that every edge meets a vertex of degree at least $\Delta(G)-1$; we prove that this is true for some minimum mitigating set of $G$, but not necessarily for every minimum mitigating set of $G$.

\end{abstract}
\vspace{0.2cm}
{\footnotesize{2010 Mathematics Subject Classification: 05C15, 05C70
}}

\section {Introduction}
Throughout this paper all graphs have at least one edge and are simple, that is, with no loops or multiple edges. Let $G$ be a graph with vertex set ${V}(G)$ and edge set ${E}(G)$.
The degree of a vertex $v$ is denoted by $d_G(v)$.
The {\it core} of $G$, denoted by $G_\Delta$, is the subgraph of $G$ induced by the vertices of degree $\Delta(G)$, where $\Delta(G)$ is the maximum degree of $G$.
We denote the number of vertices of degree $i$ by $t_i$.
The graph $G$ is called {\it acyclic} if $G$ does not contain any cycles.
If $S \subseteq V(G)$, then the subgraph of $G$ induced by $S$ is denoted by $G[S]$.
Similarly, if $S' \subseteq E(G)$, then the subgraph of $G$ induced by $S'$ is denoted by $G[S']$.

A {\it proper edge coloring} (resp.\ vertex coloring) assigns colors to the edges (resp.\ vertices) of a graph so that no two incident edges (resp.\ adjacent vertices) have the same color.
The {\it chromatic index} of $G$, $\chi^\prime(G)$ is the minimum number of colors required for a proper edge coloring of $G$.
The {\it chromatic number} of a graph $G$, $\chi(G)$ is similarly defined as the minimum number of colors required for a proper vertex coloring of $G$.
A {\it matching} $M$ in $G$ is a set of pairwise non-incident edges.
The {\it matching number} $\alpha^\prime(G)$ of a graph $G$ is the size of a maximum matching.

The {\it chromatic edge stability number} $\mathrm{es}_{\chi}(G)$ of a graph $G$ is the minimum number of edges whose removal results in a graph $H$ with $\chi(H)=\chi(G)-1$. This concept was first introduced in the 1980's by Staton \cite{S}, and some inequalities and sharp bounds on $\mathrm{es}_\chi$ were proved in \cite{AKMN}. Similarly, the {\it chromatic vertex stability number} $vs_\chi(G)$ of a graph $G$, defined as the minimum number of vertices whose removal results in a graph $H$ with $\chi(H)=\chi(G)-1$, was first studied in \cite{ABKM}. 

A {\it mitigating set} is a set of edges $F$ in $G$ such that
$\chi^\prime(G \setminus F) < \chi^\prime(G)$.

In this paper, we focus on the {\it chromatic edge stability index} of $G$ denoted by $\mathrm{es}_{\chi^\prime}(G)$, defined as the minimum number of edges whose removal results in a graph $H$ with ${\chi^\prime(H)=\chi^\prime(G)-1}$. This concept was first studied in \cite{ABBGHM}.

As an example, consider the {\it Petersen graph}, $P$, in which $\chi^\prime(P)=4$, see  \cite[p.~7]{SSTF}. As was proved in \cite[p.~16]{SSTF}, removing one vertex of $P$ does not reduce the chromatic index. Hence, the removal of one edge cannot reduce the chromatic index. However, one can reduce the chromatic index by removing two edges, as illustrated in Figure \ref{fig:Petersen}. Hence, $\mathrm{es}_{\chi^\prime}(P) = 2$.

\begin{figure}\centering
\begin{tikzpicture}[thick]

    \foreach \pos/\name in {{(18:2)/a}, {(90:2)/b}, {(162:2)/c},
	{(234:2)/d}, {(306:2)/e}, 
	{(18:1)/f}, {(90:1)/g}, {(162:1)/h},
	{(234:1)/i}, {(306:1)/j}}
\node[vertex] (\name) at \pos {};
\draw (a) -- node[pos=0.5,fill=white,draw=white,rectangle] {$a$} (b); 
\draw (b) -- node[pos=0.5,fill=white,draw=white,rectangle] {$b$} (c); 
\draw (c) -- node[pos=0.5,fill=white,draw=white,rectangle] {$a$} (d); 
\draw (d) -- node[pos=0.5,fill=white,draw=white,rectangle] {$c$} (e); 
\draw (e) -- node[pos=0.5,fill=white,draw=white,rectangle] {$b$} (a); 

\draw (f) -- node[pos=0.5,fill=white,draw=white,rectangle] {$a$} (h); 
\draw (i) -- node[pos=0.3,fill=white,draw=white,rectangle] {$a$} (g); 
\draw (g) -- node[pos=0.7,fill=white,draw=white,rectangle] {$b$} (j); 

\draw (b) -- node[pos=0.5,fill=white,draw=white,rectangle] {$c$} (g); 
\draw (a) -- node[pos=0.5,fill=white,draw=white,rectangle] {$c$} (f); 
\draw (e) -- node[pos=0.5,fill=white,draw=white,rectangle] {$a$} (j); 
\draw (i) -- node[pos=0.5,fill=white,draw=white,rectangle] {$b$} (d); 
\draw (h) -- node[pos=0.5,fill=white,draw=white,rectangle] {$c$} (c); 
\end{tikzpicture}
\caption{A $3$-edge-colorable graph obtained by removing two edges from $P$.}\label{fig:Petersen}
\end{figure}
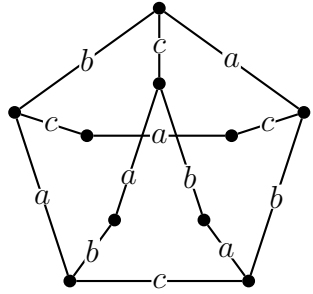

We now recall Vizing's Theorem \cite{V}.
\begin{theorem}\cite{V}
For a graph $G$, ${\Delta(G)\leq\chi^\prime(G)\leq\Delta(G)+1}$.\label{Vizing}
\end{theorem}
A graph $G$ is {\it Class $1$} if ${\chi^\prime(G)=\Delta(G)}$ and it is {\it Class $2$} if ${\chi^\prime(G)=\Delta(G)+1}$.
The following theorem was first stated by Fournier \cite{F}.
\begin{theorem}\cite{F}\label{acyclic class 1}
Let $G$ be a graph. If $G_\Delta$ is acyclic, then $G$ is Class 1.
\end{theorem}

\begin{theorem}\cite{CH}\label{core size 3}
Let $G$ be a connected graph with ${|G_\Delta| = 3}$. Then $G$ is Class 2 if and only if for some integer $n$, $G$ is obtained from $K_{2n+1}$ by removing $n-1$ independent edges.
\end{theorem}

\section{Main Results}
We now state our first theorem regarding the chromatic edge stability index.

\begin{theorem}\label{thm:class2}
Let $G$ be a graph of Class $2$. Then, 
${\mathrm{es}_{\chi^\prime}(G)\leqslant {\frac{t_{\Delta}-1}{2}}}$.
Furthermore, if $G$ is connected and $t_\Delta$ is odd, then  equality holds if and only if $G$ is obtained from $K_{2n+1}$ by removing $s$ independent edges for some integers $n$ and $s$, ${0 \leq s \leq n-1}$.
\end{theorem}

\begin{proof}
Note that it is sufficient to remove a set of edges $F$ such that $|F|\leq \frac{t_{\Delta}-1}{2}$ and ${G\setminus F}$ is Class 1, since then we have ${\chi'(G\setminus F)=\Delta(G\setminus F)\leq\Delta(G)<\chi'(G)}$.

By Theorem \ref{acyclic class 1}, since $G$ is Class 2 it contains a cycle of vertices of degree $\Delta(G)$; removing an edge from that cycle reduces the number of vertices of degree $\Delta(G)$ by $2$. Continue in this manner until either the remaining graph is Class 1 or ${\lfloor\frac{t_{\Delta}-1}{2}}\rfloor$ edges have been removed. In the latter case the number of vertices of degree $\Delta(G)$ remaining is at most $2$, so there is no cycle of such vertices and Theorem \ref{acyclic class 1} implies the remaining graph is Class 1. Thus ${\mathrm{es}_{\chi^\prime}(G)\leqslant \frac{t_{\Delta}-1}{2}}$.

Now, we prove the second part of the theorem. Let 
$A$ be the set of vertices of maximum degree. First, assume that ${\mathrm{es}_{\chi'}(G) = \frac{t_{\Delta}-1}{2}}$. Consider a maximum matching $M$ in $G_\Delta$ and denote by $V_M$ the set of vertices saturated by $M$. Remove the edges of $M$ from $G$ and call the new graph $G'$. Since $t_\Delta$ is odd, $A \neq V_M$. Note that $G'_\Delta = G[A \setminus V_M]$ has no edges and by Theorem \ref{acyclic class 1} $G'$ is Class 1. Hence, ${\mathrm{es}_{\chi'}(G)} \leqslant \alpha'(G_\Delta)$.
Notice that we have
${\alpha'(G_\Delta) \leq \frac{t_\Delta - 1}{2}}$. By combining these two inequalities and our initial assumption, we obtain ${\alpha'(G_\Delta) = \frac{t_\Delta - 1}{2}}$. Hence, ${|A \setminus V_M| = 1}$. Suppose that ${A \setminus V_M = \{v\}}$. We claim that $v$ is adjacent to all vertices of $V_M$. By contradiction, assume ${uv \notin E(G_\Delta)}$ for some ${u \in V_M}$. Suppose $uw \in M$. Remove all edges in $M$ except $uw$ from $G$ and call the new graph $G''$.
Note that ${V(G''_{\Delta}) = \{u,v,w\}}$ and $G''_\Delta$ is a forest so by Theorem \ref{acyclic class 1}, $G''$ is Class 1. Hence, 
${\chi'(G'') = \Delta(G'') = \Delta(G) < \chi'(G)}$ and $G''$ is obtained by removing $\mathrm{es}_{\chi'}(G) - 1$ edges from $G$, a contradiction. 
Similarly, for every vertex ${u \in V_M}$, if $uw \in M$, by considering the matching $M \setminus \{uw\} \cup \{vw\}$, it can be proved that $u$ is adjacent to all vertices in $A$.
Therefore, 
${G_{\Delta} \cong K_{t_\Delta}}$. Now, remove $\alpha'(G_\Delta) - 1$ edges of $M$ from $G$ and call the remaining graph $H$. Since ${\mathrm{es}_{\chi'}(G) - 1}$ edges are removed from $G$, ${\chi'(H) = \chi'(G)}$. Note that ${\Delta(H) = \Delta(G)}$. Hence, we have
${\chi'(H) = \chi'(G) = \Delta(G) + 1 = \Delta(H) + 1}$ and $H$ is Class 2. Since 
$H_\Delta \cong K_3$ and $H$ is connected, by Theorem \ref{core size 3}, $H$ is obtained from $K_{2n+1}$ by removing $n-1$ independent edges, for some integer $n$. It is straightforward to see that $G$ is obtained from $K_{2n+1}$ by removing 
$s$ independent edges, where ${s = (n - 1) - (\alpha'(G_\Delta) - 1) = n - \alpha'(G_\Delta) = n - \mathrm{es}_{\chi'}(G) \leq n - 1}$.

Now, assume that $G$ is obtained from $K_{2n+1}$ by removing a matching $M$, where $|M|=s$ and ${0 \leq s \leq n-1}$. It is not hard to see that $G$ is Class 2. 
Note that ${\frac{t_\Delta - 1}{2} = \frac{2n + 1 - 2s - 1}{2} = n - s}$. Therefore, by the first part of the theorem, we have ${\mathrm{es}_{\chi'}(G) \leq n - s}$.
Now, we prove that ${\mathrm{es}_{\chi'}(G) \geq n - s}$. Suppose $\chi'(G)$ is reduced by removing at most $n-s-1$ edges. Then there is a subgraph $G'$ with $\chi'(G')=2n$ and $|E(G)|\leq |E(G')|+n-s-1$. Consider an edge-coloring of $G'$ with $2n$ colors. Each color class is a matching, so contains at most $n$ edges. Thus $E(G')\leq 2n^2$, giving
\[|E(G)|\leq 2n^2 + n - s - 1 = \binom{2n+1}{2} - s - 1 =|E(G)|-1,\]
a contradiction.  Hence, ${\mathrm{es}_{\chi'}(G) = n - s = \frac{t_\Delta - 1}{2}}$ and the proof is completed.
\end{proof}

\begin{remark}\label{equality-class2}
For every value of $t_\Delta \geq 3$, the upper bound ${\lfloor\frac{t_{\Delta}-1}{2}\rfloor}$ can be attained for $\mathrm{es}_{\chi'}(G)$. If $t_\Delta$ is odd, this follows from the equality case of Theorem \ref{thm:class2}. If ${t_\Delta=2k}$ is even, construct $G$ as follows: start from $K_{2k+1}$ and remove an edge $xy$, then add a new vertex $z$ and the edge $xz$. Note that ${\Delta(G)=2k}$ and there are $2k$ vertices of degree $2k$ (all except $y$ and $z$). In a proper edge-coloring of a subgraph of $G$, there are at most $k+1$ edges of the color used on $xz$, and at most $k$ edges of any other color. Thus any ${H\subset G}$ with ${\chi'(H)=2k}$ satisfies ${|E(H)|\leq 2k^2+1=|E(G)|-(k-1)}$, and so ${\mathrm{es}_{\chi^\prime}(G)\geqslant k-1}$.
\end{remark}

\begin{remark}
In contrast to Class 2, for every function $f(t_\Delta)$ we can find a connected graph $G$ of Class 1 such that ${\mathrm{es}_{\chi^\prime}(G)>f(t_\Delta)}$. Consider the Petersen graph, remove one of its edges, and call the remaining graph $Q$, as shown in Figure \ref{fig:Q}.

\begin{figure}[htp]\centering
\begin{tikzpicture}[scale = 0.65, auto, swap]
\foreach \b in {1} {
      \draw (\b+0,2)node{} -- (\b+2, 2)node{};
      \draw (\b+2,2) -- (\b+4, 2)node{};
      \draw (\b+0,-2)node{} -- (\b+2, -2)node{};
      \draw (\b+2,-2) -- (\b+4, -2)node{};
      \draw (\b+1,1)node{} -- (\b+3, 1)node{};
      \draw (\b+0,0)node{} -- (\b+0,2);
      \draw (\b+0,0) -- (\b+0,-2);
      \draw (\b+4,0)node{} -- (\b+4,2);
      \draw (\b+4,0) -- (\b+4,-2);
      \draw (\b+2,2) -- (\b+2,-2);
      \draw (\b+0,2) -- (\b+4,-2);
      \draw (\b+0,-2) -- (\b+4,2);
  };
\end{tikzpicture}
\caption{The graph $Q\cong P\setminus\{e\}$.}\label{fig:Q}
\end{figure}
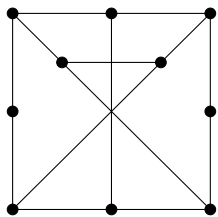

The graph $G$ illustrated in Figure \ref{fig:chain} is a chain of copies of $Q$ joined to a copy of $K_{1,3}$. Note that $\chi^\prime(G)=4$, $t_{\Delta}=1$ and $\Delta(G) = 4$. As was previously stated, one must remove at least one more edge from each copy of $Q$ in order to reduce the chromatic index. Therefore, one can arbitrarily increase $\mathrm{es}_{\chi^\prime}(G)$ by adding more copies of $Q$ to the chain.

\begin{figure}\centering
\begin{tikzpicture}[thick,scale=0.5]
  \draw (-1,0)node{} -- (0,0)node{};
  \draw (-1,1)node{} -- (0,0);
  \draw (-1,-1)node{} -- (0,0);
  
  \foreach \x in {1, 5, 12} {
    \node [black, draw=black, fill=white] at (\x+1.5, 0) {$\hspace{0.55cm} Q \hspace{0.55cm}$};
  };
  \draw (1,0) -- (0,0);
  \draw (5,0) -- (4,0);
  \draw (9,0) -- (8,0);
  \draw (11,0) -- (12,0);
   \node [black, draw=white, fill=white] at (10, 0) {\dots};
\end{tikzpicture}
\caption{A graph $G$ with $t_\Delta=1$ and $\mathrm{es}_{\chi'}(G)$ arbitrarily large.}\label{fig:chain}
\end{figure}
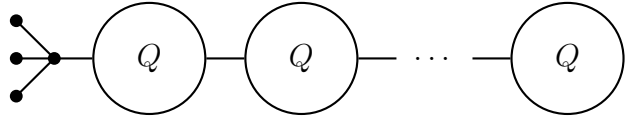
\end{remark}
In the next result, we provide an upper bound for the chromatic edge stability index in terms of $t_{\Delta-1}$ and $t_\Delta$.
\begin{theorem}\label{es for all graphs}
For every graph $G$, $\mathrm{es}_{\chi^\prime}(G)\leq t_\Delta+\frac{t_{\Delta-1}}{2}$. Furthermore, equality holds if and only if $t_{\Delta-1}=0$ and $G_\Delta$ has no edges.
\end{theorem}
\begin{proof}
We prove the following slightly stronger statement by induction. Let $H$ be a graph with maximum degree at most $d$. 
Let $t$ be the number of vertices of degree $d$, and let $s$ be the number of vertices of degree $d-1$ that have a neighbor of degree at least $d-1$. Then we may remove at most $k$ edges from $H$ to leave a graph of chromatic index at most $d-1$, where $k=t$ if $s=0$ and the vertices of degree $d$ form an independent set, and $k=t+\lfloor\frac{s-1}{2}\rfloor$ otherwise.

We prove this by induction on $t$. If $t=0$ and $s=0$, the statement follows immediately from Theorem \ref{Vizing}. If $t=0$ and $s\geq 1$ then $H$ has maximum degree $d-1$ and it suffices to remove $\lfloor\frac{s-1}{2}\rfloor$ edges to leave a graph of Class 1. We do this as in the proof of Theorem \ref{thm:class2}: remove edges between vertices of degree $d-1$ until at most one such edge remains, leaving a graph of Class 1 by Theorem \ref{acyclic class 1}. All such edges lie within the set $A$ of vertices of degree $d-1$ meeting another vertex of degree $d-1$, initially $|A|=s$, and each edge removed reduces $|A|$ by $2$, so after at most $\lfloor\frac{s-1}{2}\rfloor$ edges are removed we have $|A|\leq 2$ and so $H[A]$ has at most one edge, as required.

Now suppose $t>0$. If there is an edge $uv$ between two vertices of degree $d$, remove that edge. If not, let $w$ be an arbitrary vertex of degree $d$. If $w$ has a neighbor $x$ of degree $d-1$, remove the edge $wx$. Otherwise, choose an arbitrary neighbor $y$ and remove the edge $wy$. In each case we have removed one edge to leave a graph $H'$; let $t'$ be the number of vertices of $H'$ having degree $d$, and $s'$ be the number of vertices of degree $d-1$ that are adjacent to some other vertex of degree at least $d-1$. We now apply the induction hypothesis to $H'$.

In the first case, we have $t'=t-2$ and $s'\leq s+2$. We may therefore remove at most $t-2+\lfloor\frac{s+1}{2}\rfloor$ edges from $H'$ to leave a graph $H''$ with $\chi'(H'')\leq d-1$; together with the edge $uv$, we have removed at most $t+\lfloor\frac{s-1}{2}\rfloor$ edges from $H$ to obtain $H''$. 

In the second case, we have $t'=t-1$ and $s'\leq s$, since the set of vertices of degree $d-1$ with a neighbor of degree at least $d-1$ may gain $w$, but must lose $x$. We may therefore remove at most $k'$ edges from $H'$ to leave a graph $H''$ with $\chi'(H'')\leq d-1$, where $k'=t'=t-1$ if $s'=0$ and $k'=t'+\lfloor\frac{s'-1}{2}\rfloor$ if $s'\geq 1$. It follows that $k'\leq k-1$ and so we have removed at most $k$ edges from $H$ to obtain $H''$. 

In the final case, we have $t'=t-1$ and $s'=s$, since $d_{H'}(w)=d-1$ but $w$ has no neighbor of degree at least $d-1$, so is not counted in $s'$. Thus we may remove at most $k'=k-1$ edges from $H'$ to leave a graph $H''$ with $\chi'(H'')\leq d-1$, as required.

This completes the proof of the stronger statement. The required upper bounds follow by taking $H=G$ and $d=\Delta$, since $\chi'(G)\geq\Delta$. Note that we obtain the stronger bound of $t_\Delta+\frac{t_{\Delta-1}-1}{2}$ unless $t_{\Delta-1}=s=0$ and $G_\Delta$ has no edges.

It remains to show that equality holds if $t_{\Delta-1}=0$ and $G_\Delta$ has no edges. In this case, by Theorem \ref{acyclic class 1}, we have $\chi'(G)=\Delta$. 
In order to reduce the chromatic index to $\Delta-1$, we must remove at least one edge meeting every vertex of degree $\Delta$. Since no edge meets more than one of these vertices, we have $\mathrm{es}_{\chi'}(G)\geq t_\Delta$, as required.
\end{proof}

\begin{remark}
Theorem \ref{es for all graphs} gives $\mathrm{es}_{\chi'}(G)\leq t_\Delta+\lfloor\frac{t_{\Delta-1}-1}{2}\rfloor$ whenever $t_{\Delta-1}\geq 1$. In fact this is sharp for any values of $t=t_\Delta\geq 1$ and $s=t_{\Delta-1}\geq 1$. To see this (for $s\geq 3$), let $H_1$ be a Class 2 graph having $s$ vertices of maximum degree, such that $\mathrm{es}_{\chi'}(H_1)={\lfloor\frac{s-1}{2}\rfloor}$ (see Remark \ref{equality-class2}). Let $H_2$ be a graph with $t$ vertices of degree $\Delta(H_2)=\Delta(H_1)+1$, forming an independent set, and none of degree $\Delta(H_1)$ (for example, a disjoint union of stars). Then the disjoint union $G=H_1\cup H_2$ has $\chi'(G)=\Delta(G)$, $t_\Delta=t$, $t_{\Delta-1}=s$, and at least $t$ edges from $H_2$ and $\lfloor\frac{s-1}{2}\rfloor$ edges from $H_2$ must be removed to reduce its chromatic index. If $s\leq 2$ then the bound reduces to $t_\Delta$, so the graph $H_1$ should be omitted (and $\Delta(H_2)$ may be chosen arbitrarily).
\end{remark}

Now, we provide a formula for the chromatic edge stability index of bipartite graphs in terms of $t_{\Delta}$ and the matching number of $G_\Delta$.

\begin{theorem}
For every bipartite graph $G$, $\mathrm{es}_{\chi^\prime}(G)=t_\Delta-\alpha^\prime(G_\Delta)$.
\end{theorem}

\begin{proof}
First, we prove that ${\mathrm{es}_{\chi^\prime}(G) \geq t_\Delta-\alpha^\prime(G_\Delta)}$. Assume that $S'$ is the smallest mitigating set of $G$. Note that since $G$ is bipartite, it is Class 1 \cite{K}. Hence, in order to reduce $\chi^\prime(G)$, it is necessary and sufficient that there are no vertices of degree $\Delta(G)$ left in the remaining graph.
We now state two claims:
\\
\textbf{Claim 1.}
$G[S']$ is acyclic. 
By contradiction, assume that $G[S']$ contains a cycle $C$. Consider an edge $e = uv$ in $C$ and let $G^\prime$ be the graph obtained by removing the edges of the path $C-e$ from $G$. Note that $d_{G^\prime}(u)$ and $d_{G^\prime}(v)$ do not exceed $\Delta(G)-1$. Hence, by removing $S' - e$ from $G$, $\chi'(G)$ is reduced, a contradiction.
\\
\textbf{Claim 2.}
$G[S']$ contains no path of length 3.
Similarly to the previous claim, if $G[S']$ contains a path of length 3, by removing the middle edge of such a path from $S'$, we get a subset of the edges of size $|S'|-1$ whose removal reduces $\chi^\prime(G)$, a contradiction.

By the claims above, one can conclude that each component of $G[S']$ is the graph $K_{1,r}$ for some $r \geq 1$.

Assume that the components of $G[S']$ are ${S_1, \ldots, S_k}$. If $S_i$ is a component with at least two edges and $l$ is a leaf of $S_i$, then ${d_G(l) = \Delta(G)}$, because otherwise by removing the edge incident with $l$ from $S'$, we get a subset of edges of size $|S'|-1$ whose removal reduces $\chi^\prime(G)$, a contradiction. By a similar argument, if $S_i$ is a single edge, at least one of its endpoints has degree $\Delta(G)$ in $G$. Now, assume that in exactly $r$ components of $G[S']$, all vertices have degree $\Delta(G)$ in $G$. Hence, if we take one edge of each of these components, we get a matching of size $r$ in $G_\Delta$. Therefore, ${\alpha^\prime(G_\Delta) \geq r}$. Note that $G[S']$ contains all vertices of degree $\Delta(G)$ in $G$. Since $t_\Delta=\sum_{i=1}^{k}{|E(S_i)|} +r$, we have the following:
\[\mathrm{es}_{\chi^\prime}(G)=|S'| = \sum_{i=1}^{k}{|E(S_i)|} = t_\Delta - r \geq t_\Delta - \alpha^\prime(G_\Delta).\]

Now, we prove that ${\mathrm{es}_{\chi^\prime}(G) \leq t_\Delta-\alpha^\prime(G_\Delta)}$.
 Consider a maximum matching $M$ in $G_\Delta$. There are $t_\Delta - 2|M|$ vertices of degree $\Delta(G)$ in $G$ which are not saturated by $M$. Let $G^\prime$ be the graph obtained by removing the edges of $M$ and one incident edge of each of these vertices from $G$. Notice that ${\Delta(G^\prime) \leq \Delta(G) - 1}$ and since $G^\prime$ is Class 1, ${\chi^\prime(G^\prime) = \Delta(G^\prime) < \Delta(G) = \chi^\prime(G)}$. Hence, we get the following: \[\mathrm{es}_{\chi^\prime}(G) \leq |M| + t_\Delta - 2|M| = t_\Delta - \alpha^\prime(G_\Delta).\]
 The proof is completed.
\end{proof}

\section{Structure of minimum mitigating sets}
The mitigating sets constructed in the previous section all have the property that they only contain edges meeting vertices of degree $\Delta(G)$ or $\Delta(G)-1$. It is natural to ask whether minimum-size mitigating sets (i.e.\ sets of $\mathrm{es}_{\chi'}(G)$ edges whose removal reduces the chromatic index) necessarily have this property. We will show that the answer is ``yes'' in the sense that every graph has a minimum-size mitigating set with this property, but ``no'' in the sense that some graphs also have minimum-size mitigating sets that do not satisfy this property.

\begin{theorem}Let $G$ be any graph. Then there exists a mitigating set $S$ for $G$ of size $\mathrm{es}_{\chi'}(G)$ such that if $e=uv\in S$ then
\begin{equation}\max\{d(u),d(v)\}\geq \Delta(G)-1.
	\label{high-degree}\end{equation}
\end{theorem}
\begin{proof}
Let $T$ be any mitigating set of $G$ of size $\mathrm{es}_{\chi'}(G)$, and suppose that $T$ contains an edge $e=uv$ which does not satisfy \eqref{high-degree}. We will adapt the proof of Vizing's theorem to show that there is some edge $uv'$ with $d(v')\geq \Delta(G)-1$, such that $T\cup\{uv'\}\setminus\{uv\}$ is also a mitigating set for $G$. By a sequence of such replacements, we eventually obtain a set $S$ with the required properties.

Write $G'=G\setminus T$. Then, by minimality of $T$, $\chi'(G')=\chi'(G)-1$ but $\chi'(G'\cup\{uv\})=\chi'(G)$. Write $k$ for $\chi'(G')$; note that $k\geq\Delta(G)-1$. Fix a $k$-edge-coloring of $G'$. For any two colors $c_1$ and $c_2$, we say that a path is a \emph{$c_1c_2$-path} if its colors alternate between $c_1$ and $c_2$ (starting with either color).

We will construct a sequence $v=v_0,v_1,\ldots,v_r$ of neighbors of $u$ such that $d_G(v_r)\geq d_{G'}(v_r)\geq k$. Since $d_{G'}(u)<d_G(u)<k$, there is a color $c$ which is not used at $u$, and similarly there is a colour $c_0$ which is not used at $v_0$. However, $c_0$ must be used at $u$, and $c$ must be used at $v_0$, since otherwise we could color $uv_0$ with one of these colors to obtain a $k$-edge-coloring of $G'\cup\{uv\}$. Also, there must be a $cc_0$-path from $u$ to $v_0$, since otherwise we could swap colours in the $cc_0$ path meeting $u$ and then color $uv_0$ with $c_0$ to obtain a $k$-edge-coloring of $G'\cup\{uv\}$. 
Choose the neighbor $v_1$ such that $uv_1$ has color $c_0$. 

Suppose we have constructed a sequence $v_0,\ldots,v_i$ of distinct neighbors of $u$, and a sequence of colors $c_0,\ldots,c_{i-1}$ such that for each $j<i$ the color $c$ is used at $v_j$, the color $c_j$ is not used at $v_j$, there is a $cc_j$-path from $u$ to $v_j$, and $uv_{j+1}$ has color $c_j$. 

Suppose also that $d_{G'}(v_i)<k$. Then there is some color $c_i$ not used at $v_i$. The color $c_i$ must be used at $u$, and $c$ must be used at $v_i$, since otherwise we could recolor $uv_i$ with one of these colors, and give $uv_j$ color $c_j$ for each $j<i$ to obtain a $k$-edge-coloring of $G'\cup\{uv\}$. 

For each $i$, consider a maximal $cc_i$-path $P_i$ meeting $u$. Since $u$ does not meet $c$, the first edge of $P_i$ must be the unique edge of color $c_i$ meeting $u$, and each subsequent edge is uniquely determined since the coloring is proper and $P_i$ is maximal. Thus $P_i$ is uniquely determined. If $P_i$ does not meet $v_i$, then by swapping colors in $P_i$ and giving $uv_j$ color $c_j$ for each $j\leq i$ we obtain a $k$-edge-coloring of $G'\cup\{uv\}$. Thus $v_i$ lies on $P_i$, and since $v_i$ only uses one of the two colors it must be the end vertex of $P_i$. In particular, since $v_i\neq v_j$ it follows that $c_i\neq c_j$ for each $j<i$. Choose the neighbor $v_{i+1}$ such that $uv_{i+1}$ has color $c_i$; this ensures $v_{i+1}$ is distinct from $v_0,\ldots,v_i$. 

Thus the sequence can be extended until we have $d_{G'}(v_{r})\geq k\geq\Delta-1$. Now we may remove the edge $uv_{r}$ and give $uv_j$ color $c_j$ for each $j<r$, to obtain a $k$-edge-coloring of $G'\cup\{uv\}\setminus\{uv_{r}\}$. Thus $T\cup\{uv_{r}\}\setminus\{uv\}$ is also a mitigating set of $G$, but has fewer edges violating \eqref{high-degree}. By choosing $S$ to minimise the number of edges violating \eqref{high-degree}, we must have \eqref{high-degree} holding for all edges of $S$, as required.
\end{proof}
To complete this section, we show the following.
\begin{prop}There exists a graph $G$ having a mitigating set $S$ of size $\mathrm{es}_{\chi'}(G)$ and an edge $uv\in S$ which does not satisfy \eqref{high-degree}.
\end{prop}
\begin{proof}Construct a nine-vertex graph $G'$ as follows. Let $V(G')=A\cup B\cup C\cup\{u,v\}$, with $|A|=|C|=2$ and $|B|=3$. Add all edges between $u$ and $A$, between $A$ and $B$, between $B$ and $C$, and between $C$ and $v$. Now define a graph $G$ by adding two edges to $G'$: the edges $uv$ and $xy$, where $x$ and $y$ may be any two vertices in $A\cup B\cup C$ that are not adjacent in $G'$. (There are several possible choices for the second edge, giving three non-isomorphic possibilities for $G$.)

\begin{figure}[ht]\centering%
	\begin{tikzpicture}[thick]
		\filldraw (0,2) circle (0.05) node[draw=none,fill=none,rectangle,inner sep=5pt,anchor=north] {$u$};
		\filldraw (8,2) circle (0.05) node[draw=none,fill=none,rectangle,inner sep=5pt,anchor=north] {$v$};
		\foreach \x in {.5,3.5}{
			\filldraw (2,\x) circle (0.05);
			\filldraw (6,\x) circle (0.05);
		}
		\foreach \x in {.5,2,3.5}{
			\filldraw (4,\x) circle (0.05);
		}
		\draw (0,2) -- node[pos=0.5,fill=white,draw=white,rectangle] {$1$} (2,3.5);
		\draw (0,2) -- node[pos=0.5,fill=white,draw=white,rectangle] {$2$} (2,.5);
		\draw (2,3.5) -- node[pos=0.5,fill=white,draw=white,rectangle] {$2$} (4,3.5);
		\draw (2,3.5) -- node[pos=0.5,fill=white,draw=white,rectangle] {$3$} (4,2);
		\draw (2,3.5) -- node[pos=0.3,fill=white,draw=white,rectangle] {$4$} (4,.5);
		\draw (2,.5) -- node[pos=0.3,fill=white,draw=white,rectangle] {$1$} (4,3.5);
		\draw (2,.5) -- node[pos=0.5,fill=white,draw=white,rectangle] {$4$} (4,2);
		\draw (2,.5) -- node[pos=0.5,fill=white,draw=white,rectangle] {$3$} (4,.5);
		\draw (8,2) -- node[pos=0.5,fill=white,draw=white,rectangle] {$3$} (6,3.5);
		\draw (8,2) -- node[pos=0.5,fill=white,draw=white,rectangle] {$4$} (6,.5);
		\draw (6,3.5) -- node[pos=0.5,fill=white,draw=white,rectangle] {$4$} (4,3.5);
		\draw (6,3.5) -- node[pos=0.5,fill=white,draw=white,rectangle] {$1$} (4,2);
		\draw (6,3.5) -- node[pos=0.3,fill=white,draw=white,rectangle] {$2$} (4,.5);
		\draw (6,.5) -- node[pos=0.3,fill=white,draw=white,rectangle] {$3$} (4,3.5);
		\draw (6,.5) -- node[pos=0.5,fill=white,draw=white,rectangle] {$2$} (4,2);
		\draw (6,.5) -- node[pos=0.5,fill=white,draw=white,rectangle] {$1$} (4,.5);
		\node[draw=none,fill=none,rectangle,inner sep=5pt,anchor=south] at (4,3.5) {$x$};
		\node[draw=none,fill=none,rectangle,inner sep=5pt,anchor=north] at (4,2) {$y$};
		\draw[dashed] (4,3.5) -- (4,2);
		\draw[dashed] (0,2) to [out=90,in=180] (4,4.5) to [out=0,in=90] (8,2);
		\draw[decoration={brace,mirror,amplitude=3pt}, decorate] (1.5,0) -- (2.5,0) node [draw=none,fill=none,pos=0.5,anchor=north,yshift=-5pt] {$A$};
		\draw[decoration={brace,mirror,amplitude=3pt}, decorate] (3.5,0) -- (4.5,0) node [draw=none,fill=none,pos=0.5,anchor=north,yshift=-5pt] {$B$};
		\draw[decoration={brace,mirror,amplitude=3pt}, decorate] (5.5,0) -- (6.5,0) node [draw=none,fill=none,pos=0.5,anchor=north,yshift=-5pt] {$C$};
	\end{tikzpicture}
	\caption{The graphs $G'$ (solid edges) and one choice of $G$ (solid and dashed edges), with a $4$-edge-coloring of $G$.}\label{fig:counter}
\end{figure}
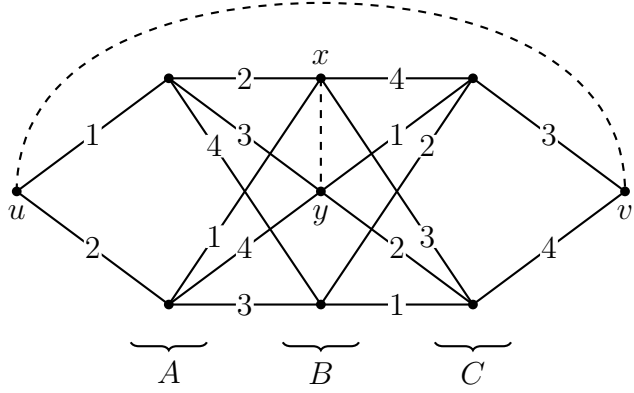

Figure \ref{fig:counter} shows a $4$-edge-colouring for $G'$. Since $\Delta(G)=5$ we clearly have $\chi'(G)\geq 5$, and indeed the colouring shown may be extended to a $5$-edge-colouring of $G$ by using a fifth color on the two extra edges.

The set $S=\{uv,xy\}$ thus shows that $\mathrm{es}_{\chi'}(G)\leq 2$, and since $uv$ does not satisfy \eqref{high-degree}, it is sufficient to show that $S$ is a minimum mitigating set, i.e.\ $\mathrm{es}_{\chi'}(G)=2$.

Suppose that this is not the case, so that $\chi'(G\setminus\{e\})=4$ for some $e\in E(G)$. In order for this to happen, certainly we need $\Delta(G\setminus\{e\})\leq 4$, and so we must have $e=xy$. Therefore $G\setminus\{e\}=G'\cup\{uv\}$ would have to be $4$-edge-colorable. Take any $4$-edge-coloring of $G'$ with colors $c_1,c_2,c_3,c_4$; let $c_1,c_2$ be the colors used at $u$. Since the vertices in $A$ have degree $4$, each color is used the same number of times on edges meeting $A$. Consequently $c_1$ and $c_2$ are used once on edges between $A$ and $B$, and $c_3$ and $c_4$ are used twice on edges between $A$ and $B$. Similarly, since each color must be used the same number of times on edges meeting $B$, $c_1$ and $c_2$ are used twice on edges between $B$ and $C$, and $c_3$ and $c_4$ are used once on edges between $B$ and $C$. Finally, since each color must be used the same number of times on edges meeting $C$, the colors used at $v$ must be $c_3$ and $c_4$. Thus there is no way to extend any $4$-edge-coloring of $G'$ to a $4$-edge-coloring of $G'\cup\{uv\}$, which completes the proof.
\end{proof}
\begin{remark}The same proof works with $|A|=|C|=k$ and $|B|=2k-1$ to give an infinite sequence of graphs with this property.\end{remark}
\section*{Acknowledgements}
The research of the first author was supported by grant number G981202 from Sharif University of Technology.
The second author was supported by the  European Research Council under the European Union's Horizon 2020 research and innovation programme (grant agreement no.\ 883810).
\end{sloppypar}

\end{document}